\documentclass[a4paper,11pt]{amsart}
\usepackage{amsthm,amssymb,mathtools,amsfonts,latexsym,rawfonts,amsmath, mathrsfs, lscape}
\usepackage{stmaryrd,tikz,pgfplots}
\usepackage{marginnote}  
\usepackage[backref=page]{hyperref}
\usepackage[top=1.5in, bottom=0.9in, left=0.9in, right=0.9in]{geometry}
\usepackage{hyperref}

\usepackage[textsize=tiny]{todonotes}

\newcommand{\Marginpar}[1]{\marginpar{\tiny{#1}}}
\newcommand{\Note}[1]{{\par\noindent\hrulefill\par\tiny{#1}\par\noindent\hrulefill\par}}
\newcommand{\Detail}[1]{{#1}}
\renewcommand{\Marginpar}[1]{}
\renewcommand{\Note}[1]{}
\renewcommand{\Detail}[1]{}
\renewcommand*{\backref}[1]{}
\renewcommand*{\backrefalt}[4]{%
    \ifcase #1 (Not cited.)%
    \or        (Cited on page~#2.)%
    \else      (Cited on pages~#2.)%
    \fi}

\usepackage{hyperref}
\usepackage[all]{xy}
\usepackage{tensor}

\usepackage{array, tabularx}

\usepackage{booktabs} 
\usepackage{siunitx}

\usepackage{tikz-cd}

\newtheorem{thm}{Theorem}
\newtheorem*{thm*}{Theorem}
\newtheorem{prop}[thm]{Proposition}
\newtheorem{lem}[thm]{Lemma}
\newtheorem*{lem*}{Lemma}
\newtheorem{cor}[thm]{Corollary}
\newtheorem*{cor*}{Corollary}

\theoremstyle{definition}
\newtheorem{defn}[thm]{Definition}
\newtheorem*{defn*}{Definition}
\newtheorem{ex}[thm]{Example}
\newtheorem{rem}[thm]{Remark}

\renewcommand{\[}{\begin{equation*}}
\renewcommand{\]}{\end{equation*}}

\newcommand\mfg{{\mathfrak g}}

\newcommand\mfu{{\mathfrak u}}

\newcommand{\bC}{\mathbb{C}}

\newcommand{\bR}{\mathbb{R}}

\newcommand{\ad}{\operatorname{ad}}

\newcommand{\del}{\partial}
\newcommand{\delbar}{\overline{\partial}}

\DeclareMathOperator\End{End}

\def \R {\mathbb R}

\def \ad {\mbox{ad}}

\renewcommand{\bar}{\overline}

\begin{document}
\parskip1mm

\title[On the invariant and anti-invariant cohomologies of hypercomplex manifolds.]{On the invariant and anti-invariant cohomologies of hypercomplex manifolds.}

\author{Mehdi Lejmi}
\address{Department of Mathematics, Bronx Community College of CUNY, Bronx, NY 10453, USA.}
\email{mehdi.lejmi@bcc.cuny.edu}
\author{Nicoletta Tardini}
\address{Dipartimento di Scienze Matematiche, Fisiche e Informatiche\\
Unit\`a di Matematica e Informatica\\
Universit\`a degli Studi di Parma\\
Parco Area delle Scienze 53/A\\
43124 Parma, Italy}
\email{nicoletta.tardini@unipr.it}

\thanks{The first author is supported by the Simons Foundation Grant \#636075. The second author is partially supported by GNSAGA of INdAM.}

\keywords{}

\subjclass[2010]{53C55 (primary); 53B35 (secondary)} 

\maketitle
\begin{abstract}
A hypercomplex structure $(I,J,K)$ on a manifold $M$ is said to be $C^\infty$-pure-and-full if the Dolbeault cohomology $H^{2,0}_{\del}(M,I)$ is the direct sum of two natural subgroups called the $\bar{J}$-invariant and the $\bar{J}$-anti-invariant subgroups.
We prove that a compact hypercomplex manifold that satisfies the quaternionic version of the $dd^c$-Lemma is $C^\infty$-pure-and-full. Moreover, we study the dimensions of the $\bar{J}$-invariant and the $\bar{J}$-anti-invariant subgroups, 
together with their analogue in the Bott-Chern cohomology. For instance, in real dimension 8, we characterize the existence of hyperk\"ahler with torsion metrics in terms of the dimension of the $\bar{J}$-invariant subgroup. 
We also study the existence of special hypercomplex structures on almost abelian solvmanifolds.
  \end{abstract}
  
  \section{introduction}
  A hypercomplex manifold $(M,I,J,K)$ is a manifold $M$ of real dimension $4n$ equipped with three complex structures $I,J,K$ satisfying the quaternionic relations. A hyperhermitian metric $g$ is a Riemannian metric on $M$ such that $I$ and $J$ are $g$-orthogonal. Now, let $\Omega$ be the $(2,0)$-form with respect to $I$ defined by $\Omega(\cdot,\cdot):=g(J\cdot,\cdot)+\sqrt{-1}g(K\cdot,\cdot)$. Then, the metric $g$ is called hyperk\"ahler if $d\Omega=0$
  and hyperk\"ahler with torsion~\cite{MR1396267} (HKT for short) if $\del\Omega=0,$ where $\del$ is the Dolbeault operator with respect to $I.$ Moreover, a hypercomplex manifold admits a unique torsion-free connection preserving $I,J,K$ called the Obata connection~\cite{MR95290}. If the holonomy of the Obata connection is in the commutator subgroup $SL(n,\mathbb{H})$ of the general quaternionic linear group $GL(n,\mathbb{H})$ then the hypercomplex manifold is said to be an $SL(n,\mathbb{H})$-manifold.
  
 On a hypercomplex manifold $(M,I,J,K)$, any $(2,0)$-form (with respect to $I$) can be expressed as the sum of the real part satisfying $J\bar{\varphi}=\varphi$, and the imaginary part satisfying $J\bar{\varphi}=-\varphi$ (here $\varphi$ is a $(2,0)$-form with respect to $I$). One can then define the following two subgroups of the Dolbeault cohomology group $H^{2,0}_{\del}(M,I)$: the \emph{$\bar J$-invariant} subgroup
 $$
 H_{\del}^{\bar J,+}(M)=\left\lbrace a\in H^{2,0}_{\del}(M,I)\,|\, \exists\varphi\in a \text{ such that } \del\varphi=0 \text{ and } J\bar{\varphi}=\varphi\right\rbrace,
 $$
 and the \emph{$\bar J$-anti-invariant subgroup}
  $$
 H_{\del}^{\bar J,-}(M)=\left\lbrace a\in H^{2,0}_{\del}(M,I)\,|\, \exists\varphi\in a \text{ such that } \del\varphi=0 \text{ and }  J\bar{\varphi}=-\varphi\right\rbrace.
 $$
 Analogue subgroups were studied in the almost-complex setting. We refer the reader for instance to~\cite{MR2574723,MR2576281,MR2817781,MR3032090,MR4244890,MR4318890}. For example, one of the questions raised in these papers is if the second de Rham cohomology group is the direct sum of those analogue subgroups. 
 
 In this paper, we study the two subgroups $H_{\del}^{\bar J,+}(M)$ and $H_{\del}^{\bar J,-}(M)$. After Preliminaries, in Section~\ref{pure-full}, we prove that if the quaternionic
version of the $dd^c$-Lemma holds on a hypercomplex manifold then {the hypercomplex structure} is $C^\infty$-pure-and-full i.e.  {$H^{2,0}_{\partial}(M,I)$}  is the direct sum of $H_{\del}^{\bar J,+}(M)$ and $H_{\del}^{\bar J,-}(M).$
\begin{thm*}(Theorem~\ref{thm:deldeljlemma-pure-full})
  {Let $(M,I,J,K)$ be a compact hypercomplex manifold that satisfies} the $\partial\partial_J$-Lemma then {the hypercomplex structure} is $C^\infty$-pure-and-full. 
 \end{thm*} 
  Then, we discuss the dimensions $h^{\pm}_{\bar{J}}$ of $H_{\del}^{\bar J,\pm}(M)$. We prove that, for a deformation of an $SL(2,\mathbb{H})$-structure on a compact manifold, the dimension $h^-_{\bar{J}}$ is an upper-semi-continuous function (Corollary~\ref{semi-cont}). This is similar to a result obtained in~\cite{MR3032090} on compact almost-complex manifolds. Then, we show that on a compact $SL(2,\mathbb{H})$-manifold, the existence of HKT metrics can be characterized in terms of $h^+_{\bar{J}}$ and the dimension of $H^{\bar{J},+}_{BC}(M)$ a subgroup of the second quaternionic Bott--Chern cohomology group. Indeed, we have the following:
 \begin{thm*}(Theorem~\ref{HKT-h+})
  On a compact $SL(2,\mathbb{H})$-manifold, either $\dim H^{\bar{J},+}_{BC}(M)=h^{+}_{\bar{J}}+1$ or $\dim H^{\bar{J},+}_{BC}(M)=h^{+}_{\bar{J}}.$ Moreover, the $SL(2,\mathbb{H})$-manifold is HKT if and only if $\dim H^{\bar{J},+}_{BC}(M)=h^{+}_{\bar{J}}.$
  \end{thm*}
  In Section~\ref{nilpotent}, we discuss $8$-dimensional hypercomplex nilmanifolds and we obtain the following:
  \begin{cor*}(Corollary~\ref{N-cor})
Let $N$ be an $8$-dimensional nilmanifold endowed with a left-invariant hypercomplex structure $(I,J,K)$, then we have the following,
\begin{itemize}
\item $N$ admits an HKT metric if and only if $h^{+}_{\bar{J}}=4$;
\item $N$ admits no HKT metrics if and only if $h^{+}_{\bar{J}}=2$.
\end{itemize}
\end{cor*}
In Section~\ref{almost-abelian}, we focus on hypercomplex almost abelian Lie groups. Such Lie groups were studied in a recent paper of Andrada and Barberis~\cite{Andrada:2022aa}.
{We similarly obtain that on unimodular almost abelian Lie groups a left-invariant hyperhermitian structure is HKT if and only if it is hyperk\"ahler. Moreover, since among hyperhermitian structures the $SL(n,\mathbb{H})$ condition plays a fundamental role
we give an explicit characterization on almost abelian solvmanifolds for an invariant hyperhermitian structure to be $SL(n,\mathbb{H})$.
Then, we focus on the $8$-dimensional case and we prove the following}

\begin{cor*}(Corollary~\ref{anti-inv-almost-abelian})
Let $\mfg$ be an $8$-dimensional non-abelian almost abelian unimodular Lie algebra equipped with a left-invariant $SL(2,\mathbb{H}) $-structure. Then the dimension of $\del$-closed non $\del$-exact left-invariant imaginary $(2,0)$-forms is non-zero if and only if $\tilde{f}= 0$ and $a=0$ where $\tilde{f}$ and $a$ are given by Theorem~\ref{thm:almost-abelian}. In particular, $\mfg$ is nilpotent and do not admit any HKT metric.
\end{cor*}

 \medskip
\noindent{\sl Acknowledgments.} Part of this work has been carried on during the stay of the second author at the Graduate Center of the City University of New York. She would like to thank Mehdi Lejmi and the City University of New York for the invitation, financial support, and hospitality. The authors are grateful to Adri\'an Andrada and Maria Laura Barberis for the private communication about the overlap with~\cite{Andrada:2022aa}.  The authors would like to thank Gueo Grantcharov, Yuri Ustinovskiy and Scott Wilson for useful discussions.
  \section{Preliminaries}

{{}    
In this section we will recall some well known facts about hypercomplex manifolds and fix some notations. 
Let $M$ be a smooth manifold and $L$ be a complex structure on $M$. Then, $L$ acts as an isomorphism on the space of $(p,q)$-forms on $M$ with respect to $L$ via
$$
L \alpha = {\left(\sqrt{-1}\right)}^{p-q}\alpha, \quad\alpha\in A^{p,q}_L(M),
$$
where we denote by $A^{p,q}_L(M)$ the space of $(p,q)$-forms on $M$ (here the bi-degree is taken with respect to the complex structure $L$).
A \emph{hypercomplex manifold} is a smooth manifold $ M$ of real dimension $4n$ equipped with three complex structures $ I,J,K $ that anticommute with each other, and such that $ IJ=K $. In particular, this induces a $2$-sphere of complex structures on $M$ given by
$$
\left\lbrace aI+bJ+cK\mid a^2+b^2+c^2=1\right\rbrace.
$$
A Riemannian metric $g$ on $ M $ that is Hermitian with respect to the three complex structures $ I,J,K $ is called \emph{hyperhermitian}.
We set $\omega_L(x,y)=g(Lx,y)$ for the fundamental form, with $L=I,J,K$.\\
One can define a $2$-form on $M$ by
$$
\Omega:=\frac{1}{2}\left(\omega_J+\sqrt{-1}\omega_K\right),
$$
and it is easy to see that
$$
\Omega\in A^{2,0}_I(M).
$$\\
Then, $(M,I,J,K,g)$ is called \emph{hyperk\"ahler} if $d\Omega=0$ and it is called \emph{hyperk\"ahler with torsion}, or briefly \emph{HKT}, if $\partial\Omega=0$, where again the complex differential operator $\partial$ is taken with respect to the complex structure $I$. 
In terms of fundamental forms, $g$ is hyperk\"ahler if and only if 
$$
d\omega_I=d\omega_J=d\omega_K=0
$$ 
and, as proven in \cite{MR1782143} it is HKT if and only if 
$$
d^{c,I}\omega_I=d^{c,J}\omega_J=d^{c,K}\omega_K,
$$
where $d^{c,L}:=L^{-1}dL$, with $L=I,J,K$.\\
Notice that if $(M,I,J,K,g)$ is a $4$-dimensional hyperhermitian manifold then it is clearly HKT for dimensional reasons but this is not true in general in higher dimension. Remark that $g$ is related to $\Omega$ by
\begin{equation}\label{eq:g-omega}
g(x,\bar y)=\Omega(x,J\bar y), \qquad x,y\in T^{1,0}_I(M).
\end{equation}
Since $JI=-IJ$ notice that $J:A^{p,q}_I(M)\to A^{q,p}_I(M)$.
We recall the following:
\begin{defn}
A form $\eta\in A^{2p,0}_I(M)$ is called \emph{real} if $J\bar\eta=\eta$. A real (2,0)-form $\eta$ is called \emph{q-positive} if $\eta(x,J\bar x)>0$, for $x\in T^{1,0}_I(M)$, $x\neq 0$.
\end{defn}
In particular, an HKT structure $\Omega$ is
\begin{itemize}
\item real $J\Omega=\bar\Omega,$
\item q-positive $\Omega(x,J\bar x)>0$, for $x\in T^{1,0}_I(M)$, $x\neq 0,$
\end{itemize}
and vice versa, a real, q-positive, $\partial$-closed $(2,0)$-form defines an HKT structure via the formula (\ref{eq:g-omega}).\\
Let $(M,I,J,K)$ be a compact $4n$-dimensional hypercomplex manifold. An important differential operator in this setting is the following
$$
\del_J:A^{p,q}_I(M)\to A^{p+1,q}_I(M),\qquad \del_J:=J^{-1}\delbar J,
$$
where the operator $\delbar$ is considered with respect to $I$.
It was shown in \cite{MR1958088} that
$$
\del_J^2=0\,,\qquad \del\del_J+\del_J\del=0.
$$
Notice that both operators increase the first degree by one, so if we fix $q=0$,  we get a cochain complex $(A^{p,0}_I(M),\del,\del_J)$ with two anticommuting differentials. 
For simplicity of notations we will drop the letter $I$ in $A^{p,q}_I(M)$ when it is understood.

\begin{defn}
On a hypercomplex manifold $(M,I,J,K)$ of real dimension $4n$, we say that the $\del\del_J$-Lemma holds if every $\del$-closed, $\del_J$-exact $(p,0)$-form in $A^{p,0}_I(M)$ is $\del\del_J$-exact, for any $0\leqslant p\leqslant 2n.$
\end{defn}
Furthermore, it is natural to consider the \emph{quaternionic Dolbeault cohomology groups}
  \[
H^{p,0}_\partial(M):=\frac{\mathrm{Ker}(\partial\vert_{A^{p,0}(M)})}{\partial A^{p-1,0}(M)}\,, \qquad H^{p,0}_{\partial_J}(M):=\frac{\mathrm{Ker}(\partial_J\vert_{A^{p,0}(M)})}{\partial_J A^{p-1,0}(M)}\,,
\]
and the  \emph{quaternionic Bott-Chern and Aeppli cohomology groups}  (see \cite{MR3709131})
\[
H^{p,0}_{\mathrm{BC}}(M):=\frac{\mathrm{Ker}(\partial\vert_{A^{p,0}(M)})\cap \mathrm{Ker}(\partial_J\vert_{A^{p,0}(M)})}{\partial\partial_J A^{p-2,0}(M)}\,,
\]
\[
H^{p,0}_{\mathrm{A}}(M):=\frac{\mathrm{Ker}(\partial\partial_J\vert_{A^{p,0}(M)})}{\partial A^{p-1,0}(M)+ \partial_J A^{p-1,0}(M)}.\,
\]
 It was shown in \cite{MR3709131} that these cohomology groups are isomorphic to the kernels of suitable elliptic differential operators and so, if $M$ is compact, they are finite dimensional. We will denote with $h^{p,0}_{\mathrm{BC}}$ the dimension of $H^{p,0}_{\mathrm{BC}}(M)$ and so on.\\
 In special bidegrees natural decompostions of forms appear. As discussed in \cite{MR3698291}, any $\varphi\in A^{2,0}_I(M)$ can be written as
 $$
 \varphi=\varphi^{\bar J,+}+\varphi^{\bar J,-},
 $$
 where
$$
\varphi^{\bar J,+}:=\frac{1}{2}\left(\varphi+J\bar \varphi \right),\qquad             \varphi^{\bar J,-}:=\frac{1}{2}\left(\varphi-J\bar \varphi \right).
$$ 
 This gives a decomposition of the bundle $\Lambda^{2,0}(M)$ in
 $$
 \Lambda^{2,0}(M)=\Lambda^{\bar J,+}(M)\oplus\Lambda^{\bar J,-}(M),
 $$
 where sections of $\Lambda^{\bar J,+}(M)$ are \emph{real} forms and are denoted by $\Omega^{\bar J,+}(M)$, and sections of $\Lambda^{\bar J,-}(M)$ satisfy $J\bar\varphi=-\varphi$ and are called \emph{imaginary} and are denoted by $\Omega^{\bar J,-}(M)$.\\
 For any compact hypercomplex manifold one can define the following two subgroups of $H^{2,0}_{\del}(M)$, the \emph{$\bar J$-invariant} subgroup
 $$
 H_{\del}^{\bar J,+}(M)=\left\lbrace a\in H^{2,0}_{\del}(M)\,|\, \exists\varphi\in a \text{ such that } \del\varphi=0 \text{ and } \varphi\in\Omega^{\bar J,+}(M)\right\rbrace,
 $$
 and the \emph{$\bar J$-anti-invariant subgroup}
  $$
 H_{\del}^{\bar J,-}(M)=\left\lbrace a\in H^{2,0}_{\del}(M)\,|\, \exists\varphi\in a \text{ such that } \del\varphi=0 \text{ and }  \varphi\in\Omega^{\bar J,-}(M)\right\rbrace\,.
 $$
 We recall the following definition (cf. \cite{MR3698291})
 \begin{defn}
 A hypercomplex structure $(I,J,K)$ on a smooth manifold $M$ is called
 \begin{itemize}
 \item \emph{$C^\infty$-pure} if
 $$
  H_{\del}^{\bar J,+}(M)\cap H_{\del}^{\bar J,-}(M)=\left\lbrace 0\right\rbrace;
 $$
  \item \emph{$C^\infty$-full} if
 $$
  H_{\del}^{\bar J,+}(M)+H_{\del}^{\bar J,-}(M)=H^{2,0}_{\del}(M);
 $$
  \item \emph{$C^\infty$-pure-and-full} if
 $$
  H_{\del}^{\bar J,+}(M)\oplus H_{\del}^{\bar J,-}(M)=H^{2,0}_{\del}(M).
 $$
 \end{itemize}
 \end{defn}}
 
 A hypercomplex manifold $(M,I,J,K)$ admits a unique torsion-free connection preserving $I,J,K$ called the Obata connection~\cite{MR95290}.
 The holonomy of the Obata connection then lies in the general quaternionic linear group $GL(n,\mathbb{H})$ (see for example~\cite{MR2959038}).
 However, in many examples such as nilmanifolds~\cite{MR2496748}, the holonomy is actually contained in $SL(n,\mathbb{H})$ (see~\cite{MR1958088,MR2384800}). 
 \begin{defn}
 A hypercomplex manifold $(M,I,J,K)$ of real dimension $4n$ is called an $SL(n,\mathbb{H})$-manifold if the holonomy of the Obata connection lies in 
 $SL(n,\mathbb{H})$.
  \end{defn}
{Verbitsky in \cite{MR2384800} proved that if a compact $4n$-dimensional hypercomplex manifold $(M,I,J,K)$ is $SL(n,\mathbb{H})$ then the canonical bundle of $(M,I)$ is holomorphically trivial. Moreover, the vice versa also holds under the additional assumption that there exists an HKT metric.}
We denote an $SL(n,\mathbb{H})$-manifold by $(M,I,J,K,\Phi)$, where $\Phi$ is a nowhere degenerate form in $A^{2n,0}_I(M)$ and we can also assume that $\Phi=J\bar{\Phi}$, in particular $\del \Phi=\del_J\Phi=0.$ It was shown in \cite{MR3698291} that every hypercomplex structure on a compact $SL(2,\mathbb{H})$-manifold is $C^\infty$-pure-and-full.

  \section{Pure and Full Hypercomplex structures}\label{pure-full}
  
  In this section we will focus on the $C^\infty$-pure-and-full condition. First of all we prove the following which is an analogue of the complex case (see~\cite{MR2576281,MR2601348} and~\cite[Theorem 2.4]{MR2817781})

   \begin{thm}\label{thm:deldeljlemma-pure-full}
   {Let $(M,I,J,K)$ be a compact hypercomplex manifold that satisfies} the $\partial\partial_J$-Lemma then {the hypercomplex structure} is $C^\infty$-pure-and-full.  \end{thm}
  \begin{proof}
  First, we prove that the hypercomplex structure is $C^\infty$-pure. Let ${\mathfrak{a}}\in H^{\bar{J},+}_{\partial}(M)\cap H^{\bar{J},-}_{\partial}(M).$
  Choose a $\partial$-closed $(2,0)$-form $\alpha\in \Omega^{\bar{J},+} (M)$ and a $\partial$-closed $(2,0)$-form $\beta\in \Omega^{\bar{J},-}(M)$ both representatives of $\mathfrak{a}$. Then, we have $\alpha-\beta=\partial\gamma$, for some $(1,0)$-form. Since $\alpha\in \Omega^{\bar{J},+}(M) $
  and  $\beta\in \Omega^{\bar{J},-}(M)$, then $\alpha$ and $\beta$ are also $\partial_J$-closed. Hence $\partial\gamma$ is $\partial_J$-closed.
  The hypercomplex structure satisfies the $\partial\partial_J$-lemma so $\partial\gamma=\partial\partial_Jw$, for some complex-valued function $w.$
  We write $w=u+\sqrt{-1}v$, for some real-valued functions $u,v.$ Hence $\alpha-\beta=\partial\partial_J\left(u+\sqrt{-1}v\right)$ and so 
  \begin{equation*}
  \alpha-\partial\partial_Ju=\beta+\sqrt{-1}\partial\partial_Jv.
  \end{equation*}
 Since $ \alpha-\partial\partial_Ju\in \Omega^{\bar{J},+}(M)$ and $\beta+\sqrt{-1}\partial\partial_Jv.\in\Omega^{\bar{J},-}(M),$ we deduce that $\alpha=\partial\partial_Ju$ and $\beta=-\sqrt{-1}\partial\partial_Jv$ and so ${\mathfrak{a}}$ is the zero class.
  
 Now, we would like to prove that the hypercomplex structure is $C^\infty$-full. Let $\alpha$ be a $\partial$-closed $(2,0)$-form representative of ${\mathfrak{a}}\in H^{2,0}_\partial(M).$ First, we claim that we can choose $\alpha$ such that
 $\partial\alpha=\partial_J\alpha=0$. Indeed, the form $\partial_J\alpha$ is $\partial$-closed and $\partial_J$-exact hence $\partial_J\alpha=\partial_J\partial\beta$, for some $1$-form $\beta.$ Hence, $\alpha-\partial\beta$ is $\partial$-closed, $\partial_J$-closed and cohomologous to $\alpha.$ Now, we decompose $\alpha$ as $\alpha=\alpha^{\bar{J},+}+\alpha^{\bar{J},-}.$ Because $\partial\alpha=\partial_J\alpha=0$, we deduce that
 $\partial\alpha^{\bar{J},+}=\partial\alpha^{\bar{J},-}=0$ and so we get the classes $[\alpha^{\bar{J},+}]\in H^{\bar{J},+}_{\partial}(M)$ and $[\alpha^{\bar{J},-}]\in H^{\bar{J},-}_{\partial}(M).$ The theorem follows.

   \end{proof}
  
  {As a consequence of Theorem \ref{thm:deldeljlemma-pure-full} and \cite{MR3709131} we have}
  
  \begin{cor}
  The hypercomplex structure on a compact HKT $SL(n,\mathbb{H})${-manifold} is $C^\infty$-pure-and-full.
  \end{cor}
  
 Now, on a compact $SL(2,\mathbb{H})$-manifold $(M,I,J,K,\Phi)$ equipped with a hyperhermitian metric $g$, we define the operator:
 \begin{align*}
 P:\Omega^{\bar{J},-}(M)&\rightarrow \Omega^{\bar{J},-}(M),\\
 \alpha&\mapsto \left(\partial\partial^\star\alpha\right)^{\bar{J},-},
 \end{align*}
 where $\left(\cdot\right)^{\bar{J},-}$ is the imaginary part, and $\partial^\star$ is defined as the adjoint of $\del$ with respect to the (global) Hermitian inner product $$\langle\alpha,\beta\rangle=\int_Mh(\alpha,\beta)\,\frac{\Omega^2\wedge\bar{\Phi}}{2},$$
 (here $2h=g-\sqrt{-1}\omega_I,$ and $\Omega$ is the $(2,0)$-form induced by $g$). Moreover, $\partial^\star=-\ast\del\ast$, where
 $\ast$ is the Hodge star operator defined by (see~\cite{MR3698291} for more details)
 $$\alpha\wedge\ast\beta\wedge\bar{\Phi}=h(\alpha,\beta)\,\frac{\Omega^2\wedge\bar{\Phi}}{2}.$$ We also define the Laplacian $\Delta_\del=\partial\partial^\star+\partial^\star\partial.$
  \begin{lem}\label{operator}
  On a compact $SL(2,\mathbb{H})$-manifold $(M,I,J,K,\Phi)$ equipped with a hyperhermitian metric $g$, the operator $P$ is a self-adjoint strongly elliptic linear operator with kernel the $\Delta_\partial$-harmonic imaginary $(2,0)$-forms.
  \end{lem}
  \begin{proof}
  If $\alpha$ is in the kernel of $P$, then $0=\langle  \left(\partial\partial^\star\alpha\right)^{\bar{J},-},\alpha\rangle=\langle  \partial^\star\alpha,\partial^\star\alpha\rangle.$ Hence $\partial^\star\alpha=\partial\alpha=0$ because $\ast \alpha=\alpha$ so $\alpha$ is  $\Delta_\partial$-harmonic. 
 Furthermore, we can express $P$ as follows
  \begin{align*}
  P(\alpha)&= \left(\partial\partial^\star\alpha\right)^{\bar{J},-},\\
  &=\frac{1}{2}\left(Id+\ast\right)\left(\partial\partial^\star\alpha\right)-\frac{1}{4}h\left(\left(Id+\ast\right)\left(\partial\partial^\star\alpha\right),\Omega    \right)\,\Omega,\\
  &=\frac{1}{2}\Delta_\partial\alpha-\frac{1}{4}h\left(\Delta_\partial\alpha,\Omega\right)\,\Omega.
  \end{align*}
  We would like to compute the principal symbol of the operator $P.$ First, we remark that
  $$h\left(\partial\partial^\star\alpha,\Omega\right)=h\left(\ast\partial\partial^\star\alpha,\ast\Omega\right)=-h\left(\ast\partial\ast\partial\ast\alpha,\Omega\right)=h\left(\partial^\star\partial\alpha,\Omega\right).$$
 A straightforward computation shows that $h\left(\Delta_\partial\alpha,\Omega\right)$ is a first order operator on $\alpha.$ Indeed,
  \begin{align*}
  \ast h\left(\Delta_\partial\alpha,\Omega\right)&=2\ast h\left(\partial^\star\partial\alpha,\Omega\right),\\
 &=- 2\ast h\left(\ast\partial\ast\partial\alpha,\Omega\right),\\
  &=-2\partial\ast\partial\alpha\wedge\Omega,\\
  &=-2\partial\left(\ast\partial\alpha\wedge\Omega\right)+2\ast\partial\alpha\wedge\partial\Omega,\\
&= -2\partial\left(L_{\Omega}\ast\partial\alpha\right)+2\ast\partial\alpha\wedge\partial\Omega,\\
&=2\partial\left(\ast\Lambda_\Omega\partial\alpha\right)+2\ast\partial\alpha\wedge\partial\Omega,\\
&=2\partial\left(\ast h\left(\partial\alpha,\Omega\right)\right)+2\ast\partial\alpha\wedge\partial\Omega,\\
&=-2\partial\left(\ast h\left(\alpha,\partial\Omega\right)\right)+2\ast\partial\alpha\wedge\partial\Omega,
  \end{align*}
where we use the fact that $h(\alpha,\Omega)=0.$ Here $L_{\Omega}$ denotes the operator $L_{\Omega}(\cdot)=\cdot\wedge\Omega$ and $\Lambda_{\Omega}:=\ast L_{\Omega}\ast$ is the contraction by $\Omega.$
We conclude that the principal symbol of $P$ is the same as $\frac{1}{2}\Delta_\partial$. The lemma follows.
  \end{proof}
 Denote by $h^{\pm}_{\bar{J}}$ the dimension of $H^{\bar{J},\pm}_{\partial}(M)$ and by $h^{2,0}_{\partial}$ the dimension of $H^{2,0}_{\partial}(M)$. Then, on a compact $SL(2,\mathbb{H})$-manifold $(M,I,J,K,\Phi)$, we have by~\cite{MR3698291} 
 \begin{equation}\label{dim-equality}
 h^{2,0}_{\partial}=h^{+}_{\bar{J}}+h^{-}_{\bar{J}}.
 \end{equation}
 We can then prove a path-wise semi-continuity property of $h^{-}_{\bar{J}}.$ This is similar to the result obtained by~\cite{MR3032090} on almost-complex manifolds. 
  \begin{cor}\label{semi-cont}
  Let $(I_t,J_t,K_t,\Phi_t)$ be a smooth family of $SL(2,\mathbb{H})$-structures on {a compact manifold} $M$ with $t\in[0,1]$. Then, $h^{-}_{\bar{J}_t}$ 
  is an upper-semi-continuous function in $t.$
  \end{cor}
  \begin{proof}
  This follows from Lemma~\ref{operator} and the upper-semi-continuity of the kernel of a family of elliptic operators~\cite[Theorem 4.3]{MR0302937}. 
  \end{proof} 
  \begin{rem}\label{rem:no-lower-semic}
  From Equation~(\ref{dim-equality}) and because the dimension $h^{2,0}_{\partial}$ depends on the choice of the complex-structure, the dimension $h^{+}_{\bar{J}}$ is not necessarily lower-semi-continuous, as we show in Example \ref{ex:deformation}.
Clearly, if the initial hypercomplex structure admits an hyperkahler metric, then along small deformations the Hodge numbers do not vary and so in that case $h^{+}_{\bar{J}}$ is lower-semi-continuous. 
  \end{rem}
  An immediate consequence of Corollary~\ref{semi-cont} is the following 
   \begin{cor}
   Let $(I_t,J_t,K_t,\Phi_t)$ be a smooth family of $SL(2,\mathbb{H})$-structures on {a compact manifold} $M$ such that $h^{-}_{\bar{J}_0}=0$ at $t=0$. Then, $h^{-}_{\bar{J}_t}=0$ for a small $t.$
   \end{cor}
  For any compact hypercomplex manifold {$(M,I,J,K)$}, we can define the following two subgroups of $H^{2,0}_{BC}(M)$: 
  $$H^{\bar{J},+}_{BC}(M):=\{\mathfrak{a}\in H^{2,0}_{BC}(M)\,|\,\exists\, \alpha\in\mathfrak{a} \text{ such that } \partial \alpha=\partial_J\alpha=0\text{ and } \alpha\in\Omega^{\bar{J},+}(M)\},$$
  $$H^{\bar{J},-}_{BC}(M):=\{\mathfrak{a}\in H^{2,0}_{BC}(M)\,|\,\exists\, \alpha \in\mathfrak{a} \text{ such that }\partial \alpha=\partial_J\alpha=0 \text{ and } \alpha\in\Omega^{\bar{J},-}(M) \}.$$
  We can easily deduce the following:
 \begin{prop}\label{BC-decomposition}
{Let $(M,I,J,K)$ be a compact hypercomplex manifold. Then}, $$H^{2,0}_{BC}(M)=H^{\bar{J},+}_{BC}\oplus H^{\bar{J},-}_{BC}.$$
  \end{prop}
  \begin{proof}
  Let ${\mathfrak{a}}\in H^{\bar{J},+}_{BC}(M)\cap H^{\bar{J},+}_{BC}(M).$ Choose $\alpha\in \Omega^{\bar{J},+}(M) $ and form $\beta\in \Omega^{\bar{J},-}(M)$ both representatives of $\mathfrak{a}$. Then  $\alpha-\beta=\partial\partial_J\left(u+\sqrt{-1}v\right)$ for some real values functions $u,v$. Hence
 $ \alpha-\partial\partial_Ju=\beta+\sqrt{-1}\partial\partial_Jv$ and so $H^{\bar{J},+}_{BC}(M)\cap H^{\bar{J},+}_{BC}(M)$ is the zero class.
 Let $\alpha$ be a $\partial$-closed and $\partial_J$-closed representative of ${\mathfrak{a}}\in H^{2,0}_{BC}(M).$ We decompose $\alpha$ as $\alpha=\alpha^{\bar{J},+}+\alpha^{\bar{J},-}$ where $\alpha^{\bar{J},\pm}\in \Omega^{\bar{J},\pm}(M)$. Then, $\partial\alpha^{\bar{J},\pm}=\partial_J\alpha^{\bar{J},\pm}=0.$ The proposition follows.
  \end{proof}

Consider the natural map $H^{2,0}_{BC}(M)\to H^{2,0}_{\partial}(M)$, we have the following
\begin{lem}
On a compact hypercomplex manifold of any dimension the natural maps $H^{\bar{J},-}_{BC}(M)\to H^{\bar{J},-}_{\partial} (M)$ and
 $H^{\bar{J},+}_{BC}(M)\to H^{\bar{J},+}_{\partial}(M)$ are surjective.

\end{lem}
\begin{proof}
 To prove that both maps are surjective, we consider $\alpha\in\Omega^{\bar{J},\pm}(M)$ a representative of $\mathfrak{a}\in H^{\bar{J},\pm}_{\partial}(M)$.
  Because $\alpha\in\Omega^{\bar{J},\pm}(M)$, then $\partial\alpha=\partial_J\alpha=0$. The lemma follows.
\end{proof}
  On a compact $SL(2,\mathbb{H})$-manifold, we can push further these relations and prove the following
  \begin{lem}\label{maps}
  On a compact $SL(2,\mathbb{H})$-manifold, the natural map $H^{\bar{J},-}_{BC}(M)\mapsto H^{\bar{J},-}_{\partial}(M) $ is an isomorphism.
  \end{lem}
  \begin{proof}
We only need to prove that the map $H^{\bar{J},-}_{BC}(M)\mapsto H^{\bar{J},-}_{\partial}(M) $ is injective. Let $\alpha\in\Omega^{\bar{J},-}(M)$ be a representative of $\mathfrak{a}\in H^{\bar{J},-}_{BC}(M)$. We suppose
 that $\alpha=\partial \beta$, for some $(1,0)$-form $\beta$. Then,
 \begin{eqnarray*}
 \|\alpha\|^2_h&=&\int_M\alpha\wedge\ast\alpha\wedge\bar{\Phi},\\
 &=&\int_M\alpha\wedge\alpha\wedge\bar{\Phi},\\
&=&\int_M\partial\beta\wedge\partial\beta\wedge\bar{\Phi}=0.
  \end{eqnarray*}
  We deduce that $\alpha=0$ and hence the map is injective.
  \end{proof}

{On  $SL(2,\mathbb{H})$-manifolds we are then able to characterize the existence of HKT metrics in terms of  $\dim H^{\bar{J},+}_{BC}(M)$ and $h^{+}_{\bar{J}}$ as follows.}

  \begin{thm}\label{HKT-h+}
  On a compact $SL(2,\mathbb{H})$-manifold, either $\dim H^{\bar{J},+}_{BC}(M)=h^{+}_{\bar{J}}+1$ or $\dim H^{\bar{J},+}_{BC}(M)=h^{+}_{\bar{J}}.$ Moreover, the $SL(2,\mathbb{H})$-manifold is HKT if and only if $\dim H^{\bar{J},+}_{BC}(M)=h^{+}_{\bar{J}}.$
  \end{thm}
  \begin{proof}
  From~\cite[Theorem 9.8]{MR3698291}, we have that $$0\leqslant h^{2,0}_{BC}- h^{2,0}_{\partial}\leqslant 1.$$
  Moreover, the $SL(2,\mathbb{H})$-manifold is HKT if and only if  $h^{2,0}_{BC}=h^{2,0}_{\partial}.$
  It follows from~(\ref{dim-equality}), Proposition~(\ref{BC-decomposition}) and Lemma~(\ref{maps}) that
  $$0\leqslant \dim H^{\bar{J},+}_{BC}(M)-h^{+}_{\bar{J}}\leqslant 1,$$
  and that the $SL(2,\mathbb{H})$-manifold is HKT if and only if $\dim H^{\bar{J},+}_{BC}(M)=h^{+}_{\bar{J}}.$
  \end{proof}
  
We now compute explicitly  the spaces $H^{\bar J,\pm}_{\partial}(M)$ and $H^{\bar J,\pm}_{BC}(M)$ on a family of examples.

\begin{ex}\label{ex:deformation}
Let $t\in (0,1)$ and consider the following family of Lie algebras $\mathfrak{g}_t$ (see \cite{MR2101226}) with structure equations
$$
[e_1,e_2]=-t\,e_6,\quad [e_3,e_4]=(1-t)\,e_6,
$$
$$
[e_1,e_3]=-t\,e_7,\quad [e_2,e_4]=(t-1)\,e_7,
$$
$$
[e_1,e_4]=-t\,e_8,\quad [e_2,e_3]=(1-t)\,e_8,
$$
or equivalently
$$
de^1=de^2=de^3=de^4=de^5=0,
$$
$$
de^6=te^{12}-(1-t)e^{34},
\quad
de^7=te^{13}-(t-1)e^{24},
\quad
de^6=te^{14}-(1-t)e^{23}.
$$
Define the hypercomplex structure
$$
Ie_1=e_2,\quad Ie_3=e_4,\quad Ie_5=e_6,\quad Ie_7=e_8
$$
$$
Je_1=e_3,\quad Je_2=-e_4,\quad Je_5=e_7,\quad Je_6=-e_8.
$$
The Lie algebras $\mathfrak{g}_t$ are all isomorphic to $\mathfrak{n}_3$, see \cite{MR2101226}.
Setting
$$
\varphi^1=e^1+ie^2,\quad
\varphi^2=e^3+ie^4,\quad
\varphi^3=e^5+ie^6,\quad
\varphi^4=e^7+ie^8
$$
as global coframe of $(1,0)$ forms, the complex structure equations become
$$
\left\lbrace
\begin{array}{lcl}
d\varphi^1 & =& 0,\\
d\varphi^2 & = & 0,\\
d\varphi^3 & = &- \frac{t}{2}\varphi^{1\bar1}+\frac{1-t}{2}\varphi^{1\bar2},\\
d\varphi^4 & = &\left(t-\frac{1}{2}\right)\varphi^{12}-\frac{1}{2}\varphi^{2\bar1}.\\
\end{array}
\right.
$$
{We recall that a complex structure $J$ on a Lie algebra $\mathfrak{g}$ is said to be \emph{abelian} if $[Jx,Jy]=[x,y]$ for every $x,y\in\mathfrak{g}$, this is equivalent to $d\Lambda^{1,0}(\mathfrak{g}^*_{\mathbb{C}})\subseteq \Lambda^{1,1}(\mathfrak{g}^*_{\mathbb{C}})$. In \cite{MR2496748} it is proven that an hypercomplex nilmanifold admits an HKT metric if and only if the underlying hypercomplex structure is abelian.}

Notice that, in the example, the hypercomplex structure is abelian if and only if $t=\frac{1}{2}$ and so, by \cite{MR2496748}, there exists an HKT structure if and only if $t=\frac{1}{2}$.\\
 Also remark that the simply connected nilpotent Lie groups $G_t$ associated to $\mathfrak{g}_t$ admit lattices and we will still denote the hypercomplex structure $(I,J,K)$ with the same letters. Since the complex structure $I$ is nilpotent the Dolbeault cohomology groups on the associated nilmanifolds can be computed using only invariant forms.\\
In particular, notice that for $t\neq \frac{1}{2}$, we have
$$
H^{2,0}_{\partial}(M) \simeq\left\langle\varphi^{13},\varphi^{14},\varphi^{23},\varphi^{24}\right\rangle.
$$
Moreover,
$$
H^{\bar J,+}_{\partial}(M)\simeq\left\langle\varphi^{13}+\varphi^{24},\varphi^{14}-\varphi^{23}\right\rangle
\quad\text{and}\quad
H^{\bar J,-}_{\partial}(M)\simeq\left\langle\varphi^{13}-\varphi^{24},\varphi^{14}+\varphi^{23}\right\rangle.
$$
For $t=\frac{1}{2}$, the hypercomplex structure is abelian and we have
$$
H^{2,0}_{\partial}(M) \simeq\left\langle\varphi^{12},\varphi^{13},\varphi^{14},\varphi^{23},\varphi^{24},\varphi^{34}\right\rangle.
$$
Moreover,
$$
H^{\bar J,+}_{\partial}(M)\simeq\left\langle\varphi^{13}+\varphi^{24},\varphi^{14}-\varphi^{23},\varphi^{34},\varphi^{12}\right\rangle
\quad\text{and}\quad
H^{\bar J,-}_{\partial}(M)\simeq\left\langle\varphi^{13}-\varphi^{24},\varphi^{14}+\varphi^{23}\right\rangle.
$$

{Notice that for $t=\frac{1}{2}$, $h^{\bar J,+}_{\partial}(M)=4$ and for $t\neq \frac{1}{2}$, $h^{\bar J,+}_{\partial}(M)=2$ confirming, as mentioned in Remark \ref{rem:no-lower-semic} that in general, $h^{\bar J,+}_{\partial}(M)$ is not lower-semi-continuous.}

Similarly, we can compute the quaternionic Bott-Chern cohomology using invariant forms by ~\cite{MR3698291}.
Therefore, for $t\neq \frac{1}{2}$, we have
$$
H^{2,0}_{BC}(M) \simeq\left\langle\varphi^{12},\varphi^{13},\varphi^{14},\varphi^{23},\varphi^{24}\right\rangle.
$$
Moreover,
$$
H^{\bar J,+}_{BC}(M)\simeq\left\langle \varphi^{12},\varphi^{13}+\varphi^{24},\varphi^{14}-\varphi^{23}\right\rangle
\quad\text{and}\quad
H^{\bar J,-}_{BC}(M)\simeq\left\langle\varphi^{13}-\varphi^{24},\varphi^{14}+\varphi^{23}\right\rangle.
$$
In particular, as expected, notice that $\text{dim}\,H^{\bar J,+}_{BC}(M)=\text{dim}\,H^{\bar J,+}_{\partial}(M)+1$ and $\text{dim}\,H^{\bar J,-}_{BC}(M)=\text{dim}\,H^{\bar J,-}_{\partial}(M)$.\\
For $t=\frac{1}{2}$, the hypercomplex structure is abelian and so by \cite{MR2496748} there exists an HKT balanced metric. In such a case, by \cite{Gentili:2022uj}, we know that
$$
H^{2,0}_{BC}(M)\simeq H^{2,0}_{\partial}(M),
$$
and so we can use the previous computation.

\end{ex}

\begin{rem}
{
On $4$-dimensional manifolds admitting almost-complex structures it was proved that the
almost-complex structures $J$ with  $h_J^-= 0$ (where $h_J^-$ is the dimension of the subspace of $H^2_{dR}(M)$ with classes represented by $J$-anti-invariant forms) form an open dense set in the
$\mathcal{C}^\infty$-Fr\'echet-topology in the space of almost-complex structures metric related to an integrable one \cite[Theorem 1.1]{MR3032090}. Based on this, Draghici, Li and Zhang  made a conjecture (Conjecture 2.4 in \cite{MR3032090}) about $h_J^-$ on a compact 4-manifold which asserts that  $h_J^-$ vanishes for generic almost complex structures $J$. In particular, they have confirmed their conjecture for 4-manifolds with $b^+ = 1$.
}
We notice that if $(M,I,J,K,g)$ is a compact hyperk\"ahler manifold with $b_2=4$ then $h^{\bar J,-}_{\partial}=0$. Indeed, since $(M,I,g)$ is K\"ahler we have
$$
b_2=2h^{2,0}_{\partial}+h^{1,1}_{\partial}
$$
and since it satisfies the $\partial\partial_J$-lemma
$$
b_2=2h^{2,0}_{\partial}+h^{1,1}_{\partial}=2(h^{+}_{\bar{J}}+h^{-}_{\bar{J}})+h^{1,1}_{\partial}\,.
$$
Now, since the structure is hyperk\"ahler $h^{2,0}_{\partial}\geqslant 1$ ($\Omega$ represents a non-trivial class) and $h^{1,1}_{\partial}\geqslant 1$ ($\omega_I$ represents a non-trivial class).
Hence,
$$
b_2\geqslant 3+2h^{-}_{\bar{J}},
$$
so we obtain the conclusion if $b_2=4$. In particular, in this case $h^{2,0}_{\partial}=1$.
Notice that in the literature there are several results concerning the second Betti number of a hyperk\"ahler manifold, see for instance \cite{MR1879810,MR4498833}.\\
Clearly, this is a special case of having an $SL(n,\mathbb{H})$ HKT manifold $(M,I,J,K,\Phi,\Omega)$ with $h^{2,0}_{\partial}=1$. Indeed, in such a case
$$
h^{2,0}_{\partial}=h^{+}_{\bar{J}}+h^{-}_{\bar{J}}
$$
and $\Omega$ represents a non-trivial class in $H^{\bar J,+}_{\partial}$ and so $h^{-}_{\bar{J}}=0$ and $h^{+}_{\bar{J}}=1$.\\
Moreover, notice that if  $(M,I,J,K,\Phi,\Omega)$ is an $SL(2,\mathbb{H})$ HKT manifold with $h^{2,0}_{\partial}=1$ and we consider a small deformation $J_t$ of $J_0=J$, such that  $(M,I,J_t,K_t)$ is hypercomplex, then 
$$
h^{-}_{\bar{J}_t}=0
$$
since $h^{-}_{\bar{J}_t}$ is an upper-semi-continuous function of $t$, and
$$
h^{+}_{\bar{J}_t}=1
$$
since $h^{+}_{\bar{J}_t}$ is a lower-semi-continuous function of $t$ (because $I$ is fixed), and $H^{\bar J_t,+}_{\partial}(M)\subseteq H^{2,0}_{\partial}(M)$ whose dimension is $1$.

\end{rem}

\section{Hypercomplex eight-dimensional Nilpotent Lie groups}\label{nilpotent}

This section is devoted to show that $h^{+}_{\bar{J}}$ {alone} can be used to characterize the existence of HKT metrics on $8$-dimensionale nilpotent Lie groups.
We consider a nilpotent Lie algebra of real dimension $8$ equipped with a hypercomplex structure $(I,J,K)$. Then, it follows from~\cite[Proposition 3.1]{MR1997309} that we have the existence of four $d$-closed $1$-forms $e^1,e^2=Ie^1,e^3=Je^1,e^4=Ke^1.$ 

We can consider a basis of $1$-forms $\{e^1,e^2,e^3,e^4,e^5,e^6=Ie^5,e^7=Je^5,e^8=Ke^5\}$. The hypercomplex structure is given by:
$$Ie^1=e^2,\quad Ie^3=e^4,\quad Ie^5=e^6,\quad Ie^7=e^8. $$
$$Je^1=e^3,\quad Je^2=-e^4,\quad Je^5=e^7,\quad Je^6=-e^8. $$
A basis of $(1,0)$-forms is given by:
$$\varphi^1=e^1+\sqrt{-1}e^2,\quad \varphi^2=e^3+\sqrt{-1}e^4,\quad \varphi^3=e^5+\sqrt{-1}e^6,\quad \varphi^4=e^7+\sqrt{-1}e^8.$$
It follows from~~\cite{MR1997309} that
$$\partial \varphi^1=0,\quad \partial \varphi^2=0,\quad\partial \varphi^3=\frac{1}{2}\left(t_2-\sqrt{-1}t_3\right)\varphi^1\wedge\varphi^2,\quad
\partial \varphi^4=\frac{1}{2}\left(t_4+\sqrt{-1}t_1\right)\varphi^1\wedge\varphi^2,$$
for some constants $t_1,t_2,t_3,t_4.$
We consider the $(2,0)$-forms $\Phi_1,\Phi_2\in \Omega^{\bar{J},-}(M)$:
$$\Phi_1=\varphi_1\wedge\varphi_3-\varphi_2\wedge\varphi_4,\quad \Phi_1=\varphi_1\wedge\varphi_4+\varphi_2\wedge\varphi_3.$$
Then, we have that $\partial\Phi_1=\partial\Phi_2=0.$ It is clear that $\Phi_1,\Phi_2$ can not be $\partial$-exact. We conclude the following:
\begin{thm}
For any left-invariant hypercomplex structure $(I,J,K)$ on a nilpotent Lie group of real dimension $8$, $h^{-}_{\bar{J}}=2.$
\end{thm}

Moreover we have the following
\begin{thm}\label{thm:hkt-number-nilmanifold}
Let $G$ be an $8$-dimensional real nilpotent Lie group with a left-invariant hyperhermitian structure $(I,J,K,g)$, then we have the following,
\begin{itemize}
\item $\Omega$ is HKT if and only if $h^{+}_{\bar{J}}=4$;
\item $\Omega$ is not HKT if and only if $h^{+}_{\bar{J}}=2$.
\end{itemize}
\end{thm}

\begin{proof}
From \cite{MR1997309}, we consider the $(2,0)$-forms $\Psi_1,\Psi_2,\Psi_3,\Psi_4\in \Omega^{\bar{J},+}(G)$:
$$
\Psi_1=\varphi_1\wedge\varphi_2, \quad\Psi_2=\varphi_3\wedge\varphi_4,\quad
\Psi_3=\varphi_1\wedge\varphi_3+\varphi_2\wedge\varphi_4,\quad \Psi_4=\varphi_1\wedge\varphi_4-\varphi_2\wedge\varphi_3.
$$
Then, we have that $\partial\Psi_3=\partial\Psi_4=0,$ and it is clear that $\Psi_3,\Psi_4$ can not be $\partial$-exact.\\
Now, $\Omega$ is HKT if and only if the hypercomplex structure is abelian which is equivalent to $t_2-\sqrt{-1}t_3=t_4+\sqrt{-1}t_1=0$, namely
$$
t_1=t_2=t_3=t_4=0.
$$
In such a case, $\Psi_1$ and $\Psi_2$ are also $\partial$-closed and not $\partial$-exact and so $h^{+}_{\bar{J}}=4$.\\
If $\Omega$ is not HKT, then at least one among $t_1, t_2, t_3, t_4$ is not zero. Hence, $\Psi_1$ is $\partial$-exact and $\Psi_2$ is not $\partial$-closed, therefore $h^{+}_{\bar{J}}=2$.
\end{proof}

\begin{cor}\label{N-cor}
Let $N$ be an $8$-dimensional nilmanifold endowed with a left-invariant hypercomplex structure $(I,J,K)$, then we have the following,
\begin{itemize}
\item $N$ admits an HKT metric if and only if $h^{+}_{\bar{J}}=4$;
\item $N$ admits no HKT metrics if and only if $h^{+}_{\bar{J}}=2$.
\end{itemize}
\end{cor}

\begin{proof}
The result follows from \cite{MR2101226}, indeed if there exists an HKT metric then there exists also an invariant HKT metric. The equivalence follows then from Theorem \ref{thm:hkt-number-nilmanifold}.
\end{proof}

\section{Hypercomplex almost-abelian Lie groups}\label{almost-abelian}

In this section, we discuss the case of hypercomplex almost-abelian Lie groups.
{More precisely, recall that a (solvable) real Lie algebra $\mathfrak{g}$ is {\em almost-abelian} if it has a codimension one abelian ideal. A Lie group $G$ is almost-abelian if its Lie algebra is.}
 Theorem~\ref{thm:almost-abelian}, Theorem~\ref{almost-abelian-hkt}, and Theorem~\ref{almost-abelian-hyp} were obtained by Barberis and Andrada in~\cite{Andrada:2022aa}. For the sake of completeness, we give proofs that we obtained independently of~\cite{Andrada:2022aa}.
First, we recall the following fact first proven in \cite{Lauret_2015}, 

\begin{thm}\label{thm:almost-abelian}
Let $\mathfrak{g}$ be a $2m$-dimensional almost-abelian Lie algebra with codimension one abelian ideal $\mathfrak{u}$ and let $I$ be an almost complex structure on $\mathfrak{g}$. 
Choose $X\in \mathfrak{g}\setminus \mathfrak{u}$ such that $IX\in \mathfrak{u}$ and set 
$\mathfrak{u}_I\coloneqq\mathfrak{u}\cap I\mathfrak{u}$ and 
$f\coloneqq\ad_X|_{\mathfrak{u}}\in \End(\mathfrak{u})$. 
Then $I$ is integrable if and only if there are 
$f_0\in \mathfrak{gl}(\mathfrak{u}_I,I)\coloneqq\left\{h\in \End(\mathfrak{u}_I)\mid [h,I]=0\right\}\cong \mathfrak{gl}(m-1,\bC)$, $w\in \mathfrak{u}_I\cong \bR^{2m-2}$ and $a\in \bR$ such that
\begin{equation*}
f=\begin{pmatrix}
      f_0 & w \\
      \mathbf 0 & a
   \end{pmatrix}
\end{equation*}
with respect to the splitting $\mathfrak{u}=\mathfrak{u}_I\oplus \left\langle IX\right\rangle$.
\end{thm}

Now we consider the hypercomplex case

 \begin{thm}\label{thm:almost-abelian}\cite{Andrada:2022aa}
Let $\mathfrak{g}$ be a $4n$-dimensional almost-abelian Lie algebra with codimension one abelian ideal $\mathfrak{u}$ and let $I,J$ be two anti-commuting almost complex structures on $\mathfrak{g}$. 
Choose $X\in \mathfrak{g}\setminus \mathfrak{u}$ such that $IX\in \mathfrak{u}$ and set 
$\mathfrak{u}_I\coloneqq\mathfrak{u}\cap I\mathfrak{u}$, $\mathfrak{u}_{I,J}\coloneqq\mathfrak{u}\cap I\mathfrak{u}\cap J\mathfrak{u}$ and 
$f\coloneqq\ad_X|_{\mathfrak{u}}\in \End(\mathfrak{u})$. \\
Then $I, J$ are integrable if and only if there are 
$\tilde f:=f|_{\mfu_{I,J}}\in \mathrm{End}(\mfu_{I,J})$ satisfying $[\tilde f,I]=0$ and $[\tilde f,J]=0$, $v\in \mathfrak{u}_{I,J}\cong \bR^{4n-4}$ and $a\in \bR$ such that
\begin{equation*}
f=\begin{pmatrix}
      \tilde f & -Jv& Kv & v \\
      \mathbf 0 & a & 0&0\\
      \mathbf 0 & 0 &a &0\\
      \mathbf 0 & 0 &0 &a
   \end{pmatrix}\,
\end{equation*}
with respect to the splitting
$\mfu= \mfu_{I,J}\oplus \left\langle KX\right\rangle\oplus \left\langle JX\right\rangle\oplus \left\langle IX\right\rangle$, where $K:=IJ$.
\end{thm}

\begin{proof}
Set $m=2n$. From the previous result $I$ is integrable if and only if there are \\
$f_0\in \mathfrak{gl}(\mathfrak{u}_I,I)\coloneqq\left\{h\in \End(\mathfrak{u}_I)\mid [h,I]=0\right\}\cong \mathfrak{gl}(m-1,\bC)$, $w\in \mathfrak{u}_I\cong \bR^{2m-2}$ and $a\in \bR$ such that
\begin{equation*}
f=\begin{pmatrix}
      f_0 & w \\
      \mathbf 0 & a
   \end{pmatrix}
\end{equation*}
with respect to the splitting $\mathfrak{u}=\mathfrak{u}_I\oplus \left\langle IX\right\rangle$.\\
We study now the integrability of $J$, hence
we need to check when the Nijenhuis tensor
\begin{equation*}
N_J(Y,Z)=[Y,Z]+J([JY,Z]+[Y,JZ])-[JY,JZ]
\end{equation*}
vanishes. \\
First of all, notice that if $N_J(Y,Z)=0$ for some $Y,Z\in \mfg$ then $N_J(JY,Z)=N_J(Y,JZ)=N_J(JY,JZ)=0$.
Moreover, $N_J(Y,JY)=0$ for all $Y\in \mfg$. \\
Since $\mfu$ is abelian,
$N_J(Y,Z)=0$ for $Y,Z\in \mfu_J:=\mfu\cap J\mfu$.\\
According to the splitting,
$$
\mfg=\left\langle X\right\rangle\oplus \mfu_{J}\oplus \left\langle JX\right\rangle
$$
we only need to check $N_J(X,Y)=0$ for $Y\in \mfu_{J}$, Let $Y\in \mfu_{J}$, 
\begin{equation*}
N_J(X,Y)=[X,Y]+J([JX,Y]+[X,JY])-[JX,JY]=[X,Y]+J[X,JY]=f(Y)+Jf(JY),
\end{equation*}
hence
$$
N_J(X,Y)=0 \iff Jf(Y)=f(JY).
$$
In particular, we have that, since $IX\in \mfu_J$,
\begin{itemize}
\item $Jf(Y)=f(JY)$ for $Y\in \mfu_{I,J}$,
\item $Jf(IX)=-f(KX)$,
\end{itemize}
Similarly, since $I$ is integrable and $JX\in\mfu_I$, by $If(Y)=f(IY)$ for every $Y\in\mfu_I$ we obtain
$$
If(JX)=f(KX).
$$
Using the previous notations we have
$$
f(IX)=aIX+w
$$
therefore, there exist $b,c\in\mathbb{R}$ and $v\in\mfu_{I,J}$ such that
$$
f(IX)=aIX+bJX+cKX+v.
$$
By the previous considerations we have
\begin{itemize}
\item $f(KX)=-Jf(IX)$,
\item $f(JX)=-If(KX)$,
\end{itemize}
hence, from the first equation we obtain,
$$
f(KX)=-Jf(IX)=bX-cIX+aKX-Jv
$$
but $f\in\End(\mathfrak{u})$ and so $b=0$.
Similarly, from the second equation
$$
f(JX)=-If(KX)=-cX+aJX+Kv
$$
but $f\in\End(\mathfrak{u})$ and so $c=0$.
Therefore, we have
$$
f(IX)=aIX+v,
$$
$$
f(JX)=aJX+Kv,
$$
$$
f(KX)=aKX-Jv.
$$
Therefore,
$I$ and $J$ are integrable if and only if there are 
$\tilde f:=f|_{\mfu_{I,J}}\in \mathrm{End}(\mfu_{I,J})$ satisfying $[\tilde f,I]=0$ and $[\tilde f,J]=0$, $v\in \mathfrak{u}_{I,J}\cong \bR^{4n-4}$ and $a\in \bR$ such that
\begin{equation*}
f=\begin{pmatrix}
      \tilde f & -Jv& Kv & v \\
      \mathbf 0 & a & 0&0\\
      \mathbf 0 & 0 &a &0\\
      \mathbf 0 & 0 &0 &a
   \end{pmatrix}\,
\end{equation*}
with respect to the splitting
$\mfu= \mfu_{I,J}\oplus \left\langle KX\right\rangle\oplus \left\langle JX\right\rangle\oplus \left\langle IX\right\rangle$.

\end{proof}

\begin{rem}
Notice that for $n=1$, we obtain
\begin{equation*}
f=\begin{pmatrix}
a & 0&0\\
0 &a &0\\
0 &0 &a
   \end{pmatrix}\,.
\end{equation*}
In particular, $\mathfrak{g}$ is unimodular if and only if $a=0$, in such a case $\mfg=\R^3\rtimes_f \R$ is isomorphic to $\R^4$ and the associated solvmanifold is the $4$-dimensional torus, that according to the classification of $4$-dimensional hypercomplex manifold in \cite{MR915736} is the only $4$-dimensional hypercomplex solvmanifold.
\end{rem}

We study now the existence of hyperk\"ahler and HKT metrics on hypercomplex almost-abelian Lie algebras. Let $\mfg$ be a $4n$-dimensional almost-abelian Lie algebra endowed with an hypercomplex structure $(I,J,K)$ and an hyperhermitian metric $g$. Then, there exists an orthonormal basis $\left\lbrace e_i\right\rbrace$ such that $\mfu_{I,J}$ is spanned by $e_2,\ldots,e_{2n-3}$, $Ie_1=e_{2n}$, $Je_1=e_{2n-1}$, $Ke_1=-e_{2n-2}$.
In view of Theorem \ref{thm:almost-abelian}, there exist a $(4n-4)\times (4n-4)$-matrix
$\tilde A:=f|_{\mfu_{I,J}}$ satisfying $[\tilde A,I]=0$ and $[\tilde A,J]=0$, $v\in  \bR^{4n-4}$ and $a\in \bR$ such that
\begin{equation*}
f=\begin{pmatrix}
      \tilde A& -Jv& Kv & v \\
      \mathbf 0 & a & 0&0\\
      \mathbf 0 & 0 &a &0\\
      \mathbf 0 & 0 &0 &a
   \end{pmatrix}\,
\end{equation*}
with respect to the splitting
$\mfu= \mfu_{I,J}\oplus \left\langle Ke_1\right\rangle\oplus \left\langle Je_1\right\rangle\oplus \left\langle Ie_1\right\rangle$.\\
We study now the existence of an HKT metric in terms of $\tilde A,a,v$.
\begin{thm}\cite{Andrada:2022aa}\label{almost-abelian-hkt}
$(I,J,K,g)$ is HKT if and only if $\tilde A\in \mathfrak{so}(\mfu_{I,J})$ and $v=0$.
\end{thm}
\begin{proof}
We recall that $(I,J,K,g)$ is HKT if and only if
$$
d^{c,I}\omega_I=d^{c,J}\omega_J=d^{c,K}\omega_K.
$$
In particular, by \cite{MR3957836} for any $L\in\left\lbrace I,J,K\right\rbrace$, the torsion form of the Bismut connection on an almost abelian Lie algebra is given by
$$
d^{c,L}\omega_L(x,y,z)=-g([Lx,Ly],z)-g([Ly,Lz],x)-g([Lz,Lx],y),
$$
for every $x,y,z\in\mfg$.\\
Since $K=IJ$ it is enough to study when
$$
d^{c,I}\omega_I=d^{c,J}\omega_J.
$$
First of all, if $x,y,z\in\mfu_{I,J}$ then, since $\mfu$ is abelian and $\mfu_{I,J}$ is $I$- and $J$-invariant,
$$
d^{c,L}\omega_L(x,y,z)=0,
$$
for every $L\in\left\lbrace I,J\right\rbrace$.
Similarly,
\begin{itemize}
\item $d^{c,L}\omega_L(e_1,e_{2n-2},e_{2n-1})=0$,
\item $d^{c,L}\omega_L(e_1,e_{2n-2},e_{2n})=0$,
\item $d^{c,L}\omega_L(e_1,e_{2n-1},e_{2n})=0$,
\item $d^{c,L}\omega_L(e_{2n-2},e_{2n-1},e_{2n})=0$,
\item $d^{c,L}\omega_L(e_{1},y,z)=0$,
\item $d^{c,L}\omega_L(e_{2n-2},y,z)=0$,
\item $d^{c,L}\omega_L(e_1,e_{2n-2},z)=0$,
\end{itemize}
for every $L\in\left\lbrace I,J,K\right\rbrace$ and $y,z\in\mfu_{I,J}$.\\
We will compute the first one, the others are similar.
Notice that
$$
Ie_{2n-2}=-IKe_1=Je_1=e_{2n-1},\quad
Je_{2n-2}=-JKe_1=-Ie_1=-e_{2n}.
$$
Now,
$$
\begin{aligned}
d^{c,I}\omega_I(e_1,e_{2n-2},e_{2n-1})& =
-g([Ie_1,Ie_{2n-2}],e_{2n-1})-g([Ie_{2n-2},Ie_{2n-1}],e_1)-g([Ie_{2n-1},Ie_1],e_{2n-2}),\\
&= -g([e_{2n},e_{2n-1}],e_{2n-1})+g([e_{2n-1},e_{2n-2}],e_1)+g([e_{2n-2},e_{2n}],e_{2n-2}),\\
&=0,
\end{aligned}
$$
since $\mfu$ is abelian.
On the other side,
$$
\begin{aligned}
d^{c,J}\omega_J(e_1,e_{2n-2},e_{2n-1})& =
-g([Je_1,Je_{2n-2}],e_{2n-1})-g([Je_{2n-2},Je_{2n-1}],e_1)-g([Je_{2n-1},Je_1],e_{2n-2}),\\
&= g([e_{2n-1},e_{2n}],e_{2n-1})+g([e_{2n},e_{1}],e_1)-g([e_{1},e_{2n-1}],e_{2n-2}),\\
&=g([e_{2n},e_{1}],e_1)-g([e_{1},e_{2n-1}],e_{2n-2}),
\end{aligned}
$$
since $\mfu$ is abelian. Moreover, $g([e_{2n},e_{1}],e_1)=0$ since $e_1$ is orthogonal to $\mfu$, and $g([e_{1},e_{2n-1}],e_{2n-2})=0$ since
$$
[e_{1},e_{2n-1}]=ae_{2n-1}+Kv
$$
and the basis is orthogonal. Therefore, $d^{c,J}\omega_J(e_1,e_{2n-2},e_{2n-1})=0$.\\
The remaining cases give the following conditions, for every $y,z\in\mfu_{I,J}$,
\begin{itemize}
\item $d^{c,I}\omega_I(e_{2n-1},y,z)=d^{c,J}\omega_J(e_{2n-1},y,z)$ if and only if
$g(AJy,z)=g(AJz,y)$,
\item $d^{c,I}\omega_I(e_{2n},y,z)=d^{c,J}\omega_J(e_{2n},y,z)$ if and only if
$g(AIy,z)=g(AIz,y)$,
\item $d^{c,I}\omega_I(e_1,e_{2n-1},z)=d^{c,J}\omega_J(e_1,e_{2n-1},z)$ if and only if
$g(Kv,z)=0$,
\item $d^{c,I}\omega_I(e_1,e_{2n},z)=d^{c,J}\omega_J(e_1,e_{2n},z)$ if and only if
$g(v,z)=0$,
\item $d^{c,I}\omega_I(e_{2n-2},e_{2n-1},z)=d^{c,J}\omega_J(e_{2n-2},e_{2n-1},z)$ if and only if
$g(v,z)=0$,
\item $d^{c,I}\omega_I(e_{2n-2},e_{2n},z)=d^{c,J}\omega_J(e_{2n-2},e_{2n},z)$ if and only if
$g(Kv,z)=0$,
\item $d^{c,I}\omega_I(e_{2n-1},e_{2n},z)=d^{c,J}\omega_J(e_{2n-1},e_{2n},z)$ is always satisfied.
\end{itemize}

We will show explicitly that $d^{c,I}\omega_I(e_{2n-1},y,z)=d^{c,J}\omega_J(e_{2n-1},y,z)$ if and only if
$g(AJy,z)=g(AJz,y)$. First of all,
$$
\begin{aligned}
d^{c,I}\omega_I(e_{2n-1},y,z)& =-g([Ie_{2n-1},Iy],z)-g([Iy,Iz],e_{2n-1})-g([Iz,Ie_{2n-1}],y),\\
&=g([e_{2n-2},Iy],z)+g([Iz,e_{2n-2}],y),\\
&=0
\end{aligned}
$$
since $\mfu$ is abelian. On the other side,,
$$
\begin{aligned}
d^{c,J}\omega_J(e_{2n-1},y,z)& =-g([Je_{2n-1},Jy],z)-g([Jy,Jz],e_{2n-1})-g([Jz,Je_{2n-1}],y),\\
&=g([e_{1},Jy],z)+g([Jz,e_{1}],y),\\
&=g(AJy,z)+g(-AJz,y).
\end{aligned}
$$
Hence, $d^{c,I}\omega_I(e_{2n-1},y,z)=d^{c,J}\omega_J(e_{2n-1},y,z)$ if and only if
$g(AJy,z)=g(AJz,y)$.\\
Similarly, we show that $d^{c,I}\omega_I(e_1,e_{2n},z)=d^{c,J}\omega_J(e_1,e_{2n},z)$ if and only if
$g(v,z)=0$,
 First of all,
$$
\begin{aligned}
d^{c,I}\omega_I(e_1,e_{2n},z)& =-g([Ie_{1},Ie_{2n}],z)-g([Ie_{2n},Iz],e_{1})-g([Iz,Ie_{1}],e_{2n}),\\
&=g([e_{2n},e_1],z)+g([e_1,Iz],e_1),\\
&=-g(ae_{2n}+v,z)+g(AIz,e_1),\\
&=-g(v,z),
\end{aligned}
$$
and
$$
\begin{aligned}
d^{c,J}\omega_J(e_1,e_{2n},z)& =-g([Je_{1},Je_{2n}],z)-g([Je_{2n},Jz],e_{1})-g([Jz,Je_{1}],e_{2n}),\\
&=-g([e_{2n-1},e_{2n-2}],z)-g([e_{2n-2},Jz],e_1)-g([Jz,e_{2n-1}],e_{2n}),\\
&=0,
\end{aligned}
$$
hence, $d^{c,I}\omega_I(e_1,e_{2n},z)=d^{c,J}\omega_J(e_1,e_{2n},z)$ if and only if
$g(v,z)=0$.\\
Therefore, since $\mfu_{I,J}$ is $I$-, $J$-, $K$-invariant, by the previous considerations the hyperhermitian structure is HKT if and only if
\begin{itemize}
\item $g(AJy,z)=g(AJz,y)$,
\item $g(AIy,z)=g(AIz,y)$,
\item $g(v,z)=0$,
\end{itemize}
for every $y,z\in\mfu_{I,J}$. In particular, we obtain $v=0$ and, since $[A,J]=0$ on $\mfu_{I,J}$, $g(AJy,z)-g(AJz,y)=0$ for every $y,z$ if and only if
$$
g((AJ-A^tJ^t)y,z)=g((A+A^t)Jy,z)=0,
$$
for every $y,z$
if and only if $A+A^t=0$ on $\mfu_{I,J}$. This concludes the proof.
\end{proof}

The existence of an hyperk\"ahler metric can be characterized in terms of $\tilde A,a,v$.
\begin{thm}\cite{Andrada:2022aa}\label{almost-abelian-hyp}
$(I,J,K,g)$ is hyperk\"ahler if and only if $\tilde A\in \mathfrak{so}(\mfu_{I,J})$, $a=0$ and $v=0$.
\end{thm}
\begin{proof}
In view of the previous theorem $Id\omega_I=Jd\omega_J$ if and only if $\tilde A\in \mathfrak{so}(\mfu_{I,J})$ and $v=0$.\\
By \cite[Lemma 3.6]{MR4278225} $(I,g)$ is K\"ahler if and only if $v=0$ and
\begin{equation*}
\begin{pmatrix}
      \tilde A & 0& 0 \\
      \mathbf 0 & a & 0\\
      \mathbf 0 & 0 &a 
   \end{pmatrix}\,.
\end{equation*}
belongs to $\mathfrak{so}(\mfu_{I})$.\\
Therefore, if $(I,J,K,g)$ is hyperk\"ahler then in particular, since it is HKT we have $\tilde A\in \mathfrak{so}(\mfu_{I,J})$ and $v=0$, and also $(I,g)$ is K\"ahler giving $a=0$.
Viceversa, if $\tilde A\in \mathfrak{so}(\mfu_{I,J})$, $a=0$ and $v=0$ then $g$ is HKT, namely $Id\omega_I=Jd\omega_J$ and $(I,g)$ is K\"ahler that is $d\omega_I=0$, concluding the proof.
\end{proof}

As an immediate corollary we obtain
\begin{cor}
Let $\mfg$ be a unimodular $4n$-dimensional almost-abelian Lie algebra endowed with an hypercomplex structure $(I,J,K)$ and an hyperhermitian metric $g$. Then,
$(I,J,K,g)$ is HKT if and only if $(I,J,K,g)$ is hyperk\"ahler.
\end{cor}

{We characterize now the existence of $SL(n,\mathbb{H})$-structures on almost-abelian hypercomplex solvmanifolds.
In order to do so we adapt~\cite[Proposition 2.4]{MR4466741} to the hypercomplex case. }
\begin{prop}\label{prop:almost-abelian-slnh}
Let $\mfg$ be an almost-abelian Lie algebra of real dimension $4n$ equipped with a hyperhermitian structure $(I,J,K,g,\Omega)$.
Then, $(\mfg,I,J,K,g,\Omega)$ admits a closed $(2n,0)$-form if and only if $$a+\frac{1}{4}tr\left({{\tilde{f}}}\right)=0,$$
where $a,\tilde{f}$ are given by Theorem~\ref{thm:almost-abelian}.
\end{prop}
\begin{proof}
Let $\{e_1,\cdots,e_{4n}\}$ be a $g$-orthonormal basis of $\mfg$ such that $e_1$ is in $\mfg/\mfu$, $Ie_1=e_{4n}$, $Je_1=e_{4n-1}$, $Ie_{2p}=e_{2p+1}$ for $1\leq p\leq 2n-1,$ and $\{\varphi_1=e^1+\sqrt{-1}e^{4n},J\bar{\varphi_1}=e^{4n-1}-\sqrt{-1}e^{4n-2},\varphi_2=e^2+\sqrt{-1}e^3,J\bar{\varphi_2}=e^4+\sqrt{-1}e^5,\varphi_3=e^6+\sqrt{-1}e^7,J\bar{\varphi_3}=e^8+\sqrt{-1}e^9,\cdots,\varphi_{n}=e^{4n-6}+\sqrt{-1}e^{4n-5},J\bar{\varphi_n}=e^{4n-4}+\sqrt{-1}e^{4n-3}\}$ is a basis of $(1,0)$-forms.
The image of $(1,0)$-forms by $\bar{\partial}$ lies in the space of $(1,0)$-forms wedge product with $\bar{\varphi_1}$. Thus, the operator $\bar{\partial}$
acting on $(1,0)$-forms can been seen as an endomorphism of $(1,0)$-forms. With respect to the basis $\{\varphi_1,J\bar{\varphi_1},\cdots,\varphi_{n},J\bar{\varphi_n}\}$, the endomorphism is:
$$
\frac{1}{2}\left(\begin{array}{ccc}a & 0 & \tilde{v} \\0 & a & \tilde{w} \\\bold{0} & \bold{0} & {\bold{A}}^T\end{array}\right),
$$
where $a=g([e_1,Ie_1],Ie_1),\bold{A}$ is the complex matrix corresponding to $ad_{e_1}|_{\mfu_{I,J,K}}$ with respect to
the basis $\{e_2-\sqrt{-1}e_3,\cdots,e_{4n-4}-\sqrt{-1}e_{4n-3}\},\tilde{v}^T=(v_3-\sqrt{-1}v_2,\cdots,v_{4n-3}-\sqrt{-1}v_{4n-4})$
and $\tilde{w}^T=(-v_5-\sqrt{-1}v_4,v_3+\sqrt{-1}v_2,\cdots,-v_{4n-3}-\sqrt{-1}v_{4n-4},v_{4n-5}+\sqrt{-1}v_{4n-6})$ where $v_i=g([e_1,Ie_1],e_i)$ for $2\leq i\leq 4n-3$. 

Now, we compute 
\begin{eqnarray*}
\bar{\partial}\left(\varphi_1\wedge J\bar{\varphi_1}\wedge\cdots\wedge \varphi_{n}\wedge J\bar{\varphi_n}\right)&=&(a+\frac{1}{2}tr\left({\bold{A}}^T\right))\,\varphi_1\wedge\bar{\varphi_1}\wedge J\bar{\varphi_1}\wedge\cdots\wedge \varphi_{n}\wedge J\bar{\varphi_n},\\
&=&(a+\frac{1}{2}tr\left({\bold{A}}\right))\,\varphi_1\wedge\bar{\varphi_1}\wedge J\bar{\varphi_1}\wedge\cdots\wedge \varphi_{n}\wedge J\bar{\varphi_n}.
\end{eqnarray*}
We note here that $tr\left({\bold{A}}\right)$ is real. Indeed, we remark that the basis $\{e_2-\sqrt{-1}e_3,\cdots,e_{4n-4}-\sqrt{-1}e_{4n-3}\}$ can be expressed as $\{\varphi_2^\ast=e_2-\sqrt{-1}e_{3},\left(J\bar{\varphi_2}\right)^\ast=e_{4}-\sqrt{-1}e_{5},\cdots,\varphi_{n}^\ast=e_{4n-6}-\sqrt{-1}e_{4n-5},\left(J\bar{\varphi_n}\right)^\ast=e_{4n-4}-\sqrt{-1}e_{4n-3}\}$. Hence, 
\begin{eqnarray*}
tr\left({\bold{A}}\right)&=&\frac{1}{2}\sum_{i=2}^{n}g([e_1,\varphi_i^\ast],\bar{\varphi_i^\ast})+g([e_1,\left(J\bar{\varphi_i}\right)^\ast],\bar{\left(J\bar{\varphi_i}\right)^\ast})\\
&=&\frac{1}{2}\sum_{k=2,6,\cdots,4n-6}g([e_1,e_k],e_k)+g([e_1,Ie_k],Ie_k)+g([e_1,Je_k],Je_k)+g([e_1,JIe_k],JIe_k)\\
&-&\frac{1}{2}\sqrt{-1}\sum_{k=2,6,\cdots,4n-6}g([e_1,Ie_k],e_k)-g([e_1,e_k],Ie_k)-g([e_1,JIe_k],Je_k)+g([e_1,Je_k],JIe_k)\\
&=&2\sum_{k=2,6,\cdots,4n-6}g([e_1,e_k],e_k)\\
&=&\frac{1}{2}\sum_{i=2}^{4n-3}g([e_1,e_i],e_i)=\frac{1}{2}tr\left({{\tilde{f}}}\right),
\end{eqnarray*}
 where $\tilde{f}$ is given by Theorem~\ref{thm:almost-abelian}. Here, we use the fact that $[e_1,L e_k]=L[e_1,e_k]$, where $L=I,J,K.$

\end{proof}

We remark that the result is a particular case of~\cite[Proposition 2.4]{MR4466741} because the trace of the complexification of $\tilde{f}$ given by Theorem~\ref{thm:almost-abelian} is real.

As a consequence, with the notations used above, we obtain the following

\begin{cor}
Let $M:=\Gamma\backslash G$ be a $4n$-dimensional solvmanifold with $G$ almost-abelian Lie group. Let $\mathfrak{g}=\text{Lie}(G)$ and let $(I,J,K,g,\Omega)$ be an invariant hyperhermitian structure on $M$. Then, $(M,I,J,K,g,\Omega)$ is $SL(n,\mathbb{H})$ if and only if $a=0$ and $tr\left({{\tilde{f}}}\right)=0$.

\end{cor}
\begin{proof}
Since  $M:=\Gamma\backslash G$ we have that $G$ is unimodular and this is equivalent to
$$
tr\left({{\tilde{f}}}\right)=-3a.
$$
Moreover, on solvmanifolds an invariant hyperhermitian structure is $SL(n,\mathbb{H})$ if and only if there exists an invariant closed $(2n,0)$-form.
By Proposition \ref{prop:almost-abelian-slnh} this turns out to be equivalent to
 $$
 a=-\frac{1}{4}tr\left({{\tilde{f}}}\right)=\frac{3}{4}a.
 $$
 Hence we get the thesis.
\end{proof}

\subsection{Explicit construction in dimension $8$}

Let $\mfg$ be an $8$-dimensional almost-abelian Lie algebra with codimension one abelian ideal $\mathfrak{u}$ and let $I,J$ be two anti-commuting almost-complex structures on $\mathfrak{g}$. Let $\left\lbrace e_i\right\rbrace$ be a basis such that $f:=ad_{e_1}$, $\mfu=\left\langle e_2,e_3,e_4,e_5, e_6,e_7,e_8\right\rangle$
$$Ie^1=e^8,\quad Ie^2=e^3,\quad Ie^4=e^5,\quad Ie^6=e^7, $$
$$Je^1=e^7,\quad Je^2=e^4,\quad Je^3=-e^5,\quad Je^6=-e^8, $$
$$Ke^1=-e^6,\quad Je^2=e^5,\quad Ke^3=e^4,\quad Ke^7=-e^8. $$
With the above notations $\mfu_{I,J}=\left\langle e_2,e_3,e_4,e_5\right\rangle$, 
\begin{equation*}
I=\begin{pmatrix}
      0 & -1& 0 & 0 \\
     1 & 0 & 0&0\\
       0 & 0 &0 &-1\\
      0 & 0 &1 &0
   \end{pmatrix}\,,
\end{equation*}
and
\begin{equation*}
J=\begin{pmatrix}
      0 & 0& -1 & 0 \\
     0& 0 & 0&1\\
       1 & 0 &0 &0\\
      0 & -1 &0&0
   \end{pmatrix}\,.
\end{equation*}
Therefore, a real matrix $A$ commutes with $I$ and $J$ if and only if it is of the form
\begin{equation*}
A=\begin{pmatrix}
      a_{11} & -a_{21}& a_{13} & -a_{23} \\
     a_{21} & a_{11}& a_{23} &a_{13}\\
       -a_{13} & -a_{23} &a_{11} & a_{21}\\
      a_{23} & -a_{13} & -a_{21} &a_{11}
   \end{pmatrix}\,.
\end{equation*}
Assume that $I$ and $J$ are integrable, hence the structure equations become
$$
\begin{aligned}
[e_1,e_2]&=a_{11}e_2+a_{21}e_3-a_{13}e_4+a_{23}e_5\,,\\
[e_1,e_3]&=-a_{21}e_2+a_{11}e_3-a_{23}e_4-a_{13}e_5\,,\\
[e_1,e_4]&=a_{13}e_2+a_{23}e_3+a_{11}e_4-a_{21}e_5\,,\\
[e_1,e_5]&=-a_{23}e_2+a_{13}e_3+a_{21}e_4+a_{11}e_5\,,\\
[e_1,e_6]&=ae_6-v_4e_2+v_5e_3+v_2e_4-v_3e_5\,,\\
[e_1,e_7]&=ae_7-v_5e_2-v_4e_3+v_3e_4+v_2e_5\,,\\
[e_1,e_8]&=ae_8+v_2e_2+v_3e_3+v_4e_4+v_5e_5\,,
\end{aligned}
$$
for some $v_2,v_3,v_4,v_5\in\mathbb{R}$.\\
Notice that $\mfg$ is unimodular if and only if $3a+4a_{11}=0$.

A basis of $(1,0)$-forms is given by:
$$\varphi^1=e^1+\sqrt{-1}e^8,\quad \varphi^2=e^2+\sqrt{-1}e^3,\quad \varphi^3=e^4+\sqrt{-1}e^5,\quad \varphi^4=e^7-\sqrt{-1}e^6.$$
The structure equations in terms of the differential $\partial$ is:
$$\partial\varphi_1=0,\quad \partial\varphi_4=-\frac{a}{2}\varphi_1\wedge\varphi_4,$$
$$\partial\varphi_2=\frac{-a_{11}-\sqrt{-1}a_{21}}{2}\varphi_1\wedge\varphi_2+\frac{-a_{13}-\sqrt{-1}a_{23}}{2}\varphi_1\wedge\varphi_3+\frac{v_5+\sqrt{-1}v_4}{2}\varphi_1\wedge\varphi_4,$$
$$\partial\varphi_3=\frac{a_{13}-\sqrt{-1}a_{23}}{2}\varphi_1\wedge\varphi_2+\frac{-a_{11}+\sqrt{-1}a_{21}}{2}\varphi_1\wedge\varphi_3+\frac{-v_3-\sqrt{-1}v_2}{2}\varphi_1\wedge\varphi_4.$$
The $(2,0)$-form corresponding to the hyperhermitian metric is $\Omega=\varphi_1\wedge J\bar{\varphi_1}+\varphi_2\wedge J\bar{\varphi_2}=\varphi_1\wedge\varphi_4+\varphi_2\wedge\varphi_3.$
The HKT condition $\partial\Omega=0$ is equivalent to $a_{11}=v_2=v_3=v_4=v_5=0.$

On the other hand, the structure equations in terms of the differential $\bar{\partial}$ is:
$$\bar{\partial}\varphi_1=\frac{a}{2}\varphi_1\wedge\bar{\varphi}_1,\quad\bar{\partial}\varphi_4=-\frac{a}{2}\bar{\varphi}_1\wedge\varphi_4.$$
$$\bar{\partial}\varphi_2=\frac{-a_{11}-\sqrt{-1}a_{21}}{2}\bar{\varphi}_1\wedge\varphi_2+\frac{-a_{13}-\sqrt{-1}a_{23}}{2}\bar{\varphi}_1\wedge\varphi_3+\frac{v_5+\sqrt{-1}v_4}{2}\bar{\varphi}_1\wedge\varphi_4+\frac{v_3-\sqrt{-1}v_2}{2}\varphi_1\wedge\bar{\varphi}_1,$$
$$\bar{\partial}\varphi_3=\frac{a_{13}-\sqrt{-1}a_{23}}{2}\bar{\varphi}_1\wedge\varphi_2+\frac{-a_{11}+\sqrt{-1}a_{21}}{2}\bar{\varphi}_1\wedge\varphi_3+\frac{-v_3-\sqrt{-1}v_2}{2}\bar{\varphi_1}\wedge\varphi_4+\frac{v_5-\sqrt{-1}v_4}{2}\varphi_1\wedge\bar{\varphi}_1.$$
The $SL(2,\mathbb{H})$ condition $\bar{\partial}\left(\varphi_1\wedge\varphi_4\wedge\varphi_2\wedge\varphi_3\right)=0$ is equivalent to $a=-a_{11}.$ 

We consider the $(2,0)$-forms $\Phi_1,\Phi_2\in \Omega^{\bar{J},-}(M)$:
$$\Phi_1=\varphi_1\wedge\varphi_2-\varphi_4\wedge\varphi_3,\quad \Phi_2=\varphi_1\wedge\varphi_3-\varphi_2\wedge\varphi_4.$$ 
Then $\partial\Phi_1=0$ is equivalent to $a_{11}=a,a_{21}=a_{13}=a_{23}=0$ and $\partial\Phi_2=0$ is equivalent to $a_{11}=-a,a_{21}=a_{13}=a_{23}=0$. Remark that $\Phi_1,\Phi_2$ can not be $\partial$-exact.

From the above discussion, we conclude
\begin{cor}\label{anti-inv-almost-abelian}
Let $\mfg$ be an $8$-dimensional non-abelian almost-abelian unimodular Lie algebra equipped with a left-invariant $SL(2,\mathbb{H}) $-structure. Then the dimension of $\del$-closed non $\del$-exact left-invariant imaginary $(2,0)$-forms is non-zero if and only if $\tilde{f}= 0$ and $a=0$ where $\tilde{f}$ and $a$ are given by Theorem~\ref{thm:almost-abelian}. In particular, $\mfg$ is nilpotent and do not admit any HKT metric.
\end{cor}
\begin{rem} 
The nilpotent Lie algebra in Corollary~\ref{anti-inv-almost-abelian} corresponds to the Lie algebra
$\mfg_3$ in the notation of~~\cite{Andrada:2022aa}.
\end{rem}



 


\bibliographystyle{abbrv}

\bibliography{hypercomplex}

\end{document}